\documentclass[12pt]{amsart}

\usepackage{amsmath,color,graphicx,amsthm,amssymb, mathtools, enumitem, pgfplots, bm, placeins, float}

\newtheorem{theorem}{Theorem}[section]

\newtheorem*{theorem*}{Theorem}

\newtheorem*{problem}{Problem}

\newtheorem{lemma}[theorem]{Lemma}

\newtheorem{corollary}[theorem]{Corollary}

\newtheorem{definition}[theorem]{Definition}

\makeatletter
\@namedef{subjclassname@2020}{%
  \textup{2020} Mathematics Subject Classification}
\makeatother

\newenvironment{notation}{\vspace{.5\baselineskip} \noindent \begin{rm}{\bf Notation.}}{\end{rm}}

\newenvironment{remark}{\vspace{.5\baselineskip} \noindent \begin{rm}{\bf Remark.}}{\end{rm}}

\newtheorem{example}[theorem]{Example}

\numberwithin{equation}{section}

\newcommand{\invlim}{\underleftarrow{\lim}}
\newenvironment{extrainfo}
  {\global\setbox\extrainfobox=\vbox\bgroup\parindent=0pt }
  {\egroup}
\newsavebox\extrainfobox

\begin{document}

\title[Connected Generalized Inverse Limits and IVP]{Connected Generalized Inverse Limits and Intermediate Value Property}

\author{Tavish J. Dunn}

\address{Baylor University, Sid Richardson Building \\  1410 S. 4th St. \\ Waco, Texas 76706}

\email{tavish\_dunn@baylor.edu}


\thanks{Declarations of Interest: none.}

\begin{extrainfo}

\end{extrainfo}

\begin{abstract}

In this paper, we consider inverse limits of $[0,1]$ using upper semicontinuous set-valued functions.  We introduce two generalizations of the Intermediate Value Property and prove that inverse limits with upper semicontinuous set-valued bonding functions are connected if the bonding functions are surjective, have connected graphs, and have either generalization of the Intermediate Value Property. Examples are given to demonstrate that if any of the conditions is dropped, the result does not hold in general. An example is given to show that an inverse limit may be connected even if the bonding functions do not have either Intermediate Value Property. Further, we compare the structures of set-valued functions with the two types of the Intermediate Value Property.

\end{abstract}

\maketitle

\section{Introduction}
Inverse limits have been objects of study for decades, for both continua and dynamical systems. In 2004 and 2006, Mahavier and Ingram introduced the concept of inverse limits with upper semicontinuous (usc) set-valued functions \cite{IngramMahavier}, \cite{Mahavier}. Since then, extensive work has been done to determine conditions under which various results for classical inverse limits can be generalized to inverse limits with usc functions.

One property of classical inverse limits that does not hold in general for inverse limits of usc functions is that the inverse limit of continua is a continuum. Ingram and Mahavier found in \cite{IngramMahavier} that even inverse limits of usc functions on $[0,1]$ are compact and nonempty but not necessarily connected. Ingram presents the following problems in \cite{Ingram}:

\begin{problem}
\textbf{6.1} Characterize connectedness of inverse limits on continua with upper semicontinuous bonding functions. 
\end{problem}

\begin{problem}\textbf{6.2}
Characterize connectedness of inverse limits on continua with upper semicontinuous bonding functions on $[0,1]$. 
\end{problem}

\begin{problem}\textbf{6.3}
Find sufficient conditions that an inverse limit on continua with upper semicontinuous bonding functions be a continuum. 
\end{problem}

\begin{problem}\textbf{6.4}
Solve Problem 6.3 on $[0,1]$.
\end{problem}

In \cite{Mahavier}, Mahavier found that for a usc function $f:[0,1]\rightarrow2^{[0,1]}$, $\invlim\{[0,1],f\}$ is connected if each of its Mahavier Products
\[G_{n}(f)=\left\{(x_{0},x_{1},x_{2},\dots)\in\prod_{i\in\omega}[0,1]:x_{i}\in f(x_{i+1}) \ \forall i\leq n\right\}
\]
is connected. In \cite{IngramMahavier}, Ingram and Mahavier showed that, if for each $i\in\omega$, $f_{i}:X_{i+1}\rightarrow 2^{X_{i}}$ is a usc function such that $f_{i}(x)$ is connected for each $x\in X_{i+1}$, then $\invlim\{X_{i},f_{i}\}$ is connected. Nall also extended this in \cite{Nall} to show that $\invlim F$ is connected if $F:X\rightarrow 2^{X}$ is a surjective usc function with connected graph $G(F)$ such that $G(F)=\bigcup_{\alpha}G(f_{\alpha})$, where each $F_{\alpha}$ is usc and $F_{\alpha}(x)$ is connected for each $x\in X$.

Greenwood and Kennedy developed a characterization for the connectedness of inverse limits of usc functions on $[0,1]$ that are surjective and have connected graphs in \cite{GreenwoodKennedy} and refined the result in \cite{GreenwoodKennedy2}. Their necessary and sufficient condition for showing the inverse limit is disconnected involves finding a finite sequence of certain subsets of the graphs of consecutive bonding functions and using this sequence to construct a proper component of the inverse limit.

In Section~2 we introduce two generalized notions of the Intermediate Value Property that are applicable to set-valued functions. In Section~3 we use the Greenwood and Kennedy's characterization in \cite{GreenwoodKennedy2} to show that inverse limits of usc functions with either of these two notions are connected. The main result of the paper is the following which can be found in Theorem \ref{InvLim}:

\begin{theorem*}
For each $n\in\mathbb{N}$, let $f_{n}:[0,1]\rightarrow 2^{[0,1]}$ be a surjective, upper-semicontinuous function with the Weak Intermediate Value Property and a connected graph $G(f_{n})$. Then $\invlim\{[0,1],f_{n}\}$ is connected.
\end{theorem*}

The utility of this theorem lies in the fact that its conditions are properties of the individual bonding functions, without regard to their placement in the inverse sequence. We give examples in Section~4 to demonstrate that the assumptions in the above theorem cannot be dropped and that the Weak Intermediate Value Property is not a necessary condition for the inverse limit to be connected.

This paper is part of a broader study on the Intermediate Value Property as it relates to generalized inverse limits. Two generalized notions of the Intermediate Value Property, named the Intermediate Value Property and Weak Intermediate Value Property, are introduced.  These two properties are distinct from each other, but when applied to single-valued functions, both of these notions are equivalent to the classical Intermediate Value Property. The Weak Intermediate Value Property is sufficient to prove the main results of this paper. But in other areas of this study involving orbits of usc functions, the Intermediate Value Property leads to a generalization Sharkovskii's Theorem on the existence of periodic orbits of a function with a known periodic point \cite{OteyRyden} and a connection between periodicity in the bonding function and indecomposability of the inverse limit \cite{DunnRyden}.

\section{Preliminary Definitions and Notation}
Here a \emph{continuum} will refer to a nonempty, compact, connected metric space. For a continuum $X$, we denote the collection of nonempty compact subsets of $X$ by $2^{X}$ and denote the collection of nonempty subcontinua of $X$ by $C(X)$. The \emph{graph} of a function $f:[a,b]\rightarrow2^{[c,d]}$ is the set $G(f)=\{(x,y)\in[a,b]\times[c,d]:y\in f(x)\}$.

\begin{definition}
\rm{A function $f:[a,b]\rightarrow2^{[c,d]}$ is \emph{upper semicontinuous at $x$} if for every open set $U$ containing $f(x)$ there is an open set $V$ containing $x$ such that $f(V)\subseteq U$. The function $f$ is \emph{upper semicontinuous} if it is upper semicontinuous at every point in its domain.}
\end{definition}

It is well known from \cite{IngramMahavier} that $f$ is upper semicontinuous if and only if $G(f)$ is closed. If $x_{1},x_{2}\in [a,b]$, let $\overline{x_{1}x_{2}}$ denote the closed interval with endpoints $x_{1}$ and $x_{2}$.
For any $0\leq a\leq b \leq 1$, let $V_{[a,b]} = [a,b]\times [0,1]$.

\begin{definition}
\rm{Let $X_{0},X_{1},X_{2},\dots$ be a sequence of continua and for all $i\in\omega$ let $f_{i+1}:X_{i+1}\rightarrow 2^{X_{i}}$ be usc. The \emph{inverse limit} of the pair $\{X_{i},f_{i}\}$ is the subspace of $\prod_{i\in\omega}$ given by
\[\invlim\{X_{i},f_{i}\}=\left\{\bm{x}=(x_{0},x_{1},\dots)\in\prod_{i\in\omega}X_{i}:x_{i-1}\in f_{i}(x_{i}) \ \forall i\geq1\right\}.
\]
The spaces $X_{i}$ are called the \emph{factor spaces} of the inverse limit, and the functions $f_{i}$ are the \emph{bonding functions}.}
\end{definition}

\begin{definition}
\rm{Let $f:[a,b]\rightarrow 2^{[c,d]}$ be an upper semicontinuous function. We say $f$ has the \emph{Intermediate Value Property} if, given distinct $x_{1},x_{2}$, distinct $y_{1}\in f(x_{1})$, $y_{2}\in f(x_{2})$, and $y$ strictly between $y_{1}$ and $y_{2}$, there is some $x$ strictly between $x_{1}$ and $x_{2}$ such that $y\in f(x)$.

We say $f$ has the \emph{Weak Intermediate Value Property} if, given distinct $x_{1}$, $x_{2}$, and $y_{1}\in f(x_{1})$ there is some $y_{2}\in f(x_{2})$ such that if $y$ is between $y_{1}$ and $y_{2}$, then there is $x$ between $x_{1}$ and $x_{2}$ such that $y\in f(x)$.}
\end{definition}

Note that in the previous definition we do not specify if $x_{2}$ is larger than $x_{1}$. For the Intermediate Value Property, this makes no difference as the definition is symmetric with respect to $x_{1}$ and $x_{2}$. However in the definition of the Weak Intermediate Value Property, $y_{2}$ is dependent on $x_{1}$, $y_{1}$, and $x_{2}$, so order matters. So for a function to have the Weak Intermediate Value Property, it is necessary for the condition to hold when $x_{2}>x_{1}$ and $x_{1}>x_{2}$. An example of a surjective usc function $f:[0,1]\rightarrow2^{[0,1]}$ where the condition holds only for $x_{2}>x_{1}$ yet the inverse limit is disconnected can be found in Example \ref{NoIVP}.

\begin{notation} Let $i>0$, $\epsilon>0$, and $A_{i}=[a_{i},b_{i}]$ be a subset of $[0,1]$ for each $i$. For $j\in\{i,i-1\}$, define
\begin{align}
J_{j} &= [0,a_{j}),  &R_{i} &= (K_{i}\times [0,1])\cup Z_{i}, \nonumber \\
K_{j} &= (b_{j},1]  &TL_{i} &= T_{i}\cup L_{i}, \nonumber \\
Z_{i} &= A_{i}\times A_{i-1} ,  &TR_{i} &= T_{i}\cup R_{i}, \nonumber \\
T_{i} &= ([0,1]\times K_{i-1}) \cup Z_{i}, &BL_{i} &= B_{i}\cup L_{i} \nonumber \\
B_{i} & = ([0,1]\times J_{i-1})\cup Z_{i}, &BR_{i} &= B_{i}\cup R_{i}, \nonumber \\
L_{i} &=(J_{i}\times[0,1])\cup Z_{i}, & & \nonumber
\end{align}
$Z_{i}(\epsilon)= ((a_{i}-\epsilon,b_{i}+\epsilon)\times (a_{i-1}-\epsilon,b_{i-1}+\epsilon))\cap([0,1]\times[0,1]).$
\end{notation}

\begin{definition}
\rm{Suppose $i>0$ and $f:[0,1]\rightarrow 2^{[0,1]}$ is usc. If for each $j\in\{i,i-1\}$, $A_{i} = [a_{i},b_{i}] \subsetneq[0,1]$, either 
\begin{itemize}
\item $S\in \{BL, BR, TL, TR\}$, or 
\item $A_{i}\cap \{0,1\}=\emptyset$ and $S\in \{L_{i},R_{i}\}$, or 
\item $A_{i-1} \cap\{0,1\}=\emptyset$ and $S\in \{B_{i},T_{i}\}$,
\end{itemize}
and there exists $\epsilon >0$ and a component $C'$ of the set $G(f)\cap Z_{i}(\epsilon)$ such that $C'\subset S$, then any component $C$ of $C'\cap Z_{i}$ is an \emph{$S$-set in $G(f)$ framed by $A_{i}\times A_{i-1}$}, denoted $G(f)\sqsubset_{C}S$.}
\end{definition}

In the following definition, only condition $(1)$ is used for the purpose of proving our main result. But for the sake of completeness and to properly introduce Greenwood's and Kennedy's criterion for connected inverse limits of surjective usc functions on $[0,1]$ with connected graphs, we present the full definition of a CC-sequence.

\begin{definition} \label{C-Sequence}
\rm{Suppose for each $i>0$, $f_{i+1}:[0,1]\rightarrow 2^{[0,1]}$ is a surjective upper-semicontinuous function with a connected graph, denoted $G_{i+1}$, and $m,n\in\mathbb{N}$ are such that $m+1<n$. Suppose that there exist
\begin{itemize}
\item a closed interval $A_{i}\subsetneq[0,1]$ for each $i$, $m\leq i\leq n$, and
\item a point \[\langle p_{k}\rangle_{k\in\omega} \in\invlim\{[0,1],f_{i}\}\cap\left(\prod_{i<m}[0,1]\times \prod_{m\leq i\leq n}A_{i}\times \prod_{i>n}[0,1]\right).\]
\end{itemize}
For each $i>0$, let $C_{i}$ be the component of $G_{i}\cap Z_{i}$ containing $(p_{i},p_{i-1})$ and suppose the following properties hold:
\begin{enumerate}
\item $G_{m+1}\sqsubset_{C_{m+1}} R_{m+1}$ or $G_{m+1} \sqsubset_{C_{m+1}} L_{m+1}$; 
\item \begin{itemize}
\item if $n=m+2$, then  $G_{m+2} \sqsubset_{C_{m+2}} T_{m+2}$ if $G_{m+1} \sqsubset_{C_{m+1}} L_{m+1}$, and $G_{m+2}\sqsubset_{C_{m+2}} B_{m+2}$ if $G_{m+1}\sqsubset_{C_{m+1}}R_{m+1}$; 
\item if $n>m+2$, then $G_{m+2}\sqsubset_{C_{m+2}}BR_{m+2}$ or $G_{m+2}\sqsubset_{C_{m+2}}BL_{m+2}$ if $G_{m+1}\sqsubset_{C_{m+1}}R_{m+1}$, and $G_{m+2}\sqsubset_{C_{m+2}}TL_{m+2}$ or $G_{m+2}\sqsubset_{C_{m+2}}TR_{m+2}$ if $G_{m+1}\sqsubset_{C_{m+1}}L_{m+1}$;
\end{itemize}
\item if $m+2\leq i<n-1$, then $G_{i+1}\sqsubset_{C_{i+1}}BL_{i+1}$ or $G_{i+1}\sqsubset_{C_{i+1}}BR_{i+1}$, if $G_{i}\sqsubset_{C_{i}}BR_{i}$ or $G_{i}\sqsubset_{C_{i}}TR_{i}$, and $G_{i+1}\sqsubset_{C_{i+1}}TL_{i+1}$ or $G_{i+1}\sqsubset_{C_{i+1}}TR_{i+1}$ if $G_{i}\sqsubset_{C_{i}}BL_{i}$ or $G_{i}\sqsubset_{C_{i}}TL_{i}$;
\item if $n>m+2$, then $G_{n}\sqsubset_{C_{n}}B_{n}$ if $G_{n-1}\sqsubset_{C_{n-1}}BR_{n-1}$ or $G_{n-1}\sqsubset_{C_{n-1}}TR_{n-1}$, and $G_{n}\sqsubset_{C_{n}}T_{n}$ if $G_{n-1}\sqsubset_{C_{n-1}}BL_{n-1}$ or $G_{n-1}\sqsubset_{C_{n-1}}TL_{n-1}$.
\end{enumerate}
Then $\{f_{i}:i>0\}$ admits a \emph{component cropping sequence}, or \emph{CC-sequence}, \[\{A_{i}:m\leq i\leq n\},\] over $[m,n]$ with \emph{pivot point} $\langle p_{k}\rangle$. The collection $\{f_{i}:i>0\}$ of functions admits a CC-sequence if there exist $m,n\in\mathbb{N}$ such that $\{f_{i}:i>0\}$ admits a CC-sequence over $[m,n]$ with some pivot point.}
\end{definition}

\begin{theorem}[Greenwood, Kennedy \cite{GreenwoodKennedy2}] \label{CC}
Suppose that for each $i\geq 0$, $I_{i}$ is an interval, $f_{i+1}:I_{i+1}\rightarrow 2^{I_{i}}$ is a surjective upper-semicontinuous function, and $G(f_{i+1})$ is connected. The system admits a CC-sequence if and only if $\invlim\{I_{i},f_{i}\}$ is disconnected.
\end{theorem}

We show that if all bonding functions have the Weak Intermediate Value Property, then $(1)$ from the above definition cannot be met, i.e. that the graph of such a function contains no $L$-sets or $R$-sets. This along with Theorem \ref{CC} is sufficient to establish the main result.

\section{Showing Connectedness of Inverse Limits}

\begin{lemma}\label{Existence}
Let $f:[0,1]\rightarrow 2^{[0,1]}$ be an upper semicontinuous function with $G(f)$ connected. Then for all $a\leq b$ there is some subcontinuum $C$ of $G(f)\cap V_{[a,b]}$ such that $C\cap \{a\}\times [0,1]\neq \emptyset$ and $C\cap \{b\}\times [0,1]\neq \emptyset$.
\end{lemma}

\begin{proof}
Let \[\mathcal{L} = \{H: \text{$H$ is a component of $G(f)\cap V_{[a,b]}$ and $H\cap \{a\}\times [0,1]\neq\emptyset$}\},\] \[\mathcal{R} = \{K: \text{$K$ is a component of $G(f)\cap V_{[a,b]}$ and $K\cap \{b\}\times [0,1]\neq\emptyset$}\}.\]

Note for all $H\in\mathcal{L}$ and $K\in\mathcal{R}$, either $H\cap K=\emptyset$ or $H=K$ by the maximality of components. Also $G(f)\cap V_{[a,b]} = \left(\bigcup \mathcal{L}\right)\cup \left(\bigcup\mathcal{R}\right)$; otherwise there would be some component $D$ of $G(f)\cap V_{[a,b]}$ that does not intersect either $\{a\}\times[0,1]$ or $\{b\}\times [0,1]$. Then $D$ would be a proper component of $G(f)$, contradicting the connectedness of $G(f)$.

Suppose that $\mathcal{L}\cap \mathcal{R}=\emptyset$. Then $f(a)$ and $f(b)$ are disjoint closed subsets of $G(f)\cap V_{[a,b]}$ such that no component of $G(f)\cap V_{[a,b]}$ intersects $f(a)$ and $f(b)$. Then by the Cut-Wire Theorem, there are two disjoint closed sets $A$ and $B$ such that $G(f)\cap V_{[a,b]} = A\cup B$, $f(a)\subseteq A$, and $f(b)\subseteq B$.  As each $H\in \mathcal{L}$ and $K\in\mathcal{R}$ is connected, $H\subseteq A$ and $K\subseteq B$. Thus $A=\bigcup\mathcal{L}$ and $B = \bigcup\mathcal{R}$. Then $A\cup \left(G(f)\cap V_{[0,a]}\right)$ and $B\cup \left(G(f)\cap V_{[b,1]}\right)$ are nonempty disjoint closed sets whose union is $G(f)$, contradicting the connectedness of $G(f)$.

\end{proof}

The following theorem provides a graphical characterization of the Weak Intermediate Value Property. A similar result characterizing the Intermediate Value Property is provided in Theorem \ref{IVPIntersection} for the purpose of comparison.

\begin{theorem} \label{Intersection}
Let $f:[0,1]\rightarrow 2^{[0,1]}$ be a usc function such that $G(f)$ is connected. The following are equivalent:
\begin{enumerate}
 \item $f$ has the Weak Intermediate Value Property.
 \item For all $a\leq b$, each component of $G(f)\cap V_{[a,b]}$ intersects both $\{a\}\times [0,1]$ and $\{b\}\times [0,1]$.
 \end{enumerate}
\end{theorem}

\begin{proof}
$(1\Rightarrow 2)$ As in Lemma \ref{Existence}, let \[\mathcal{L} = \{H: \text{$H$ is a component of $G(f)\cap V_{[a,b]}$ and $H\cap \{a\}\times [0,1]\neq\emptyset$}\},\] \[\mathcal{R} = \{K: \text{$K$ is a component of $G(f)\cap V_{[a,b]}$ and $K\cap \{b\}\times [0,1]\neq\emptyset$}\}.\]
Note every component $D$ of $G(f)\cap V_{[a,b]}$ is a member of either $\mathcal{L}$ or $\mathcal{R}$; otherwise $D$ would be a proper component of $G(f)$, contradicting the assumption that $G(f)$ is connected. By contradiction, suppose there is some component $C$ of $G(f)\cap V_{[a,b]}$ that does not intersect both $\{a\}\times[0,1]$ and $\{b\}\times[0,1]$. Without loss of generality, suppose $C\cap (\{b\}\times[0,1])=\emptyset$. Note this implies $C\in\mathcal{L}$. Then $C$ and $f(b)$ are nonempty disjoint closed subsets of $G(f)\cap V_{[a,b]}$ such that no connected subset of $G(f)\cap V_{[a,b]}$ intersects both $C$ and $\{b\}\times f(b)$. Then by the Cut-Wire Theorem, there are disjoint closed sets $A$ and $B$ in $G(f)\cap V_{[a,b]}$ such that $A\cup B=G(f)\cap V_{[a,b]}$, $C\subseteq A$, and $\{b\}\times f(b)\subseteq B$. Note that by the connectedness of each $H\in\mathcal{L}$ and $K\in\mathcal{R}$, $H\subseteq A$ and $K\subseteq B$. Thus $A = \bigcup_{H\in\mathcal{L}\setminus(\mathcal{L}\cap\mathcal{R})}H$ and $B = \bigcup_{K\in\mathcal{R}}K$.

Let $x_{1} = \max\{x: (x,y)\in A\}$. By Lemma \ref{Existence}, there is some connected set $D\subseteq B$ such that $D\cap\{a\}\times [0,1] \neq\emptyset$ and $D\cap\{b\}\times[0,1]\neq\emptyset$. As $D$ is connected, there is some point $z\in f(x_{1})$ such that $(x_{1},z)\in D$. Note there is some $y\in f(x_{1})$ such that $(x_{1},y)\in A$ and either $y>z$ or $z>y$. Without loss of generality, suppose $z>y$. Let $y_{1}=\max\{y:(x_{1},y)\in A \text{ and } y<z\}$. Define $\epsilon = \min\{d(A,B),b-x_{1}\}>0$. Let $x_{2} = x_{1}+\frac{\epsilon}{2}$ and $y_{2}\in f(x_{2})$. By the construction of $x_{1}$ and $x_{2}$, $(x_{2},y_{2})\in B$. So $|y_{2}-y_{1}|>\frac{\epsilon}{2}$; otherwise $d((x_{1},y_{1}),(x_{2},y_{2}))<\epsilon$. Let $y\in (y_{1}-\frac{\epsilon}{2}, y_{1}+\frac{\epsilon}{2})$ be between $y_{1}$ and $y_{2}$ and $x$ between $x_{1}$ and $x_{2}$. By definition of $x_{1}$, $(x,y)\notin A$. But $d((x_{1},y_{1}),(x,y))<\epsilon$ so $(x,y)\notin B$. Thus for each $x$ between $x_{1}$ and $x_{2}$, $y\notin f(x)$, a contradiction.

$(2\Rightarrow 1)$ Choose $x_{1},x_{2}\in[0,1]$ and let $y_{1}\in f(x_{1})$. Suppose $x_{1}< x_{2}$. Let $C$ be a component of $G(f)\cap V_{[x_{1},x_{2}]}$ containing $(x_{1},y_{1})$. Choose $y_{2}$ such that $(x_{2},y_{2})\in C$. As $C$ is connected, $\pi_{2}(C)$ is connected where $\pi_{2}:[0,1]^{2}\rightarrow [0,1]$ is the projection map given by $\pi_{2}(x,y)=y$. Thus $\pi_{2}(C)$ is an interval containing $y_{1}$ and $y_{2}$. So for any $y$ between $y_{1}$ and $y_{2}$, there is some $x\in[x_{1},x_{2}]$ such that $(x,y)\in C$, i.e. $y\in f(x)$. The case where $x_{1}>x_{2}$ follows by a similar argument.

\end{proof}

\begin{corollary}\label{PointImage}
Lit $f:[0,1]\rightarrow2^{[0,1]}$ be usc function such that $G(f)$ is connected. If $f(x)$ is connected for every $x\in[0,1]$, then $f$ has the Weak Intermediate Value Property.
\end{corollary}

\begin{proof}
Let $0\leq a\leq b\leq1$. By Lemma \ref{Existence}, there is some component $C$ of $G(f)\cap V_{[a,b]}$ that intersects $\{a\}\times[0,1]$ and $\{b\}\times[0,1]$. Since $f(a)$ and $f(b)$ are connected, $C\cap\left(\{a\}\times[0,1]\right)=\{a\}\times f(a)$ and $C\cap\left(\{b\}\times[0,1]\right)=\{b\}\times f(b)$. Let $D$ be a component of $G(f)\cap V_{[a,b]}$. Since $G(f)$ is connected, either $D\cap\left(\{a\}\times[0,1]\right)\neq\emptyset$ or $D\cap\left(\{b\}\times[0,1]\right)\neq\emptyset$. In either case, $C\cap D\neq\emptyset$. Therefore $D=C$. Thus every component of $G(f)\cap V_{[a,b]}$ intersects $\{a\}\times[0,1]$ and $\{b\}\times[0,1]$, and $f$ has the Weak Intermediate Value Property by Theorem \ref{Intersection}.
\end{proof}

\begin{theorem}\label{LR}
Let $f:[0,1]\rightarrow 2^{[0,1]}$ be a function that is upper semicontinuous and has the Weak Intermediate Value Property, and  suppose $G(f)$ be connected. Then $G(f)$ contains no $L$-sets or $R$-sets.
\end{theorem}

\begin{proof}
We show that $G(f)$ contains no $L$-sets. That there are no $R$-sets follows by a symmetric argument. By way of contradiction, suppose there are some $A_{1} = [a,b]$ and $A_{0} = [c,d]$ such that $A_{1}\times A_{0}$ frames an $L$-set of $G(f)$. Let $\epsilon>0$ and $C'\subset L$ be a component of $G(f)\cap Z(\epsilon)$ that contains an $L$-set $C$. Then $C'$ and $C$ are connected sets that do not intersect any of $[a,b]\times [0,c)$, $[a,b]\times (d,1]$, or $(b,1]\times [0,1]$. Thus $C$ is a component of $V_{[a,b+\epsilon]}$. However by Theorem \ref{Intersection}, $C\cap \left(\{b+\epsilon \}\times[0,1]\right)\neq\emptyset$, a contradiction.
\end{proof}

\begin{theorem} \label{InvLim}
For each $n\in\mathbb{N}$, let $f_{n}:[0,1]\rightarrow 2^{[0,1]}$ be a surjective, upper-semicontinuous function with the Weak Intermediate Value Property and a connected graph $G(f_{n})$. Then $\invlim\{[0,1],f_{n}\}$ is connected.
\end{theorem}

\begin{proof}
By Theorem \ref{LR}, $G(f)$ contains no $L$-sets or $R$-sets. Then condition (1) of the definition of a CC-sequence cannot be met. It follows that the system does not admit a CC-sequence, and therefore $\invlim\{[0,1],f\}$ is connected by Theorem \ref{CC}.
\end{proof}

\begin{remark}
Jonathan Medaugh has informed the author that Theorem \ref{InvLim} would hold if $f_{n}^{-1}$, rather than $f_{n}$, were assumed to have the Weak Intermediate Value Property for each $n$. One approach to the proof of such a result would involve finite Mahavier products $G_{n}(f_{1},f_{2},\dots,f_{n})$ and the fact that $G_{n}(f_{1},f_{2},\dots, f_{n})$ and $G_{n}(f_{n}^{-1},f_{n-1}^{-1},\dots, f_{1}^{-1})$ are homeomorphic \cite[Theorem 2.11]{CharatonikRoe}.
\end{remark}

Before examining the structure of usc functions with the Weak Intermediate Property as unions of their subgraphs, we must introduce the notion of convergence in the hyperspace $2^{X}$ with the Hausdorff topology, i.e. what it means for a sequence of subsets of a metric space to converge. The following definition and three theorems can be found in \cite{Macias}.

\begin{definition}
\rm{Let $X$ be a space and $\{A_{i}\}_{i\in\omega}$ be a sequence of subsets of $X$. We define $\overline{\lim}\ A_{i}$ and $\underline{\lim}\ A_{i}$ as follows:
\[
\overline{\lim}\ A_{i}=\{x\in X: \text{for every open set $U\ni x$, $U\cap A_{i}\neq\emptyset$ for infinitely many  $i$}\}, 
\]\[
\underline{\lim}\ A_{i}=\{x\in X: \text{for every open set $U\ni x$, $U\cap A_{i}\neq\emptyset$ for cofinitely many  $i$}\}. 
\]
If $\overline{\lim}\ A_{i}=\underline{\lim}\  A_{i}$, we define $\lim \ A_{i}=\overline{\lim}\ A_{i}=\underline{\lim}\ A_{i}$.}
\end{definition}

\begin{theorem}\label{Hyperspace}
Let $X$ be a compact metric space. Then $2^{X}$ is compact.
\end{theorem}

\begin{theorem}
Let $X$ be a compact metric space. Then $C(X)$ is compact.
\end{theorem}

\begin{theorem}\label{liminf}
If $X$ is a compact metric space and if $\{E_{i}\}_{i\in\omega}$ is a sequence of connected subsets of $X$ such that $\underline{\lim}\ E_{i}\neq\emptyset$, then $\overline{\lim}\ E_{i}$  is connected.
\end{theorem}

\begin{theorem}\label{Union}
Let $f:[0,1]\rightarrow 2^{[0,1]}$ be an upper semicontinuous function such that $G(f)$ is connected. If $f$ has the Weak Intermediate Value Property, then there is a collection $\mathcal{F}$ of functions $g:[0,1]\rightarrow C([0,1])$ that satisfies the following:
\begin{enumerate}
\item Each $g\in\mathcal{F}$ is upper semicontinuous and has the Weak Intermediate Value Property.
\item $G(f) = \bigcup_{g\in\mathcal{F}}G(g)$.
\end{enumerate}
\end{theorem}

\begin{proof}
It suffices to show that for each $(x,y)\in G(f)$, there is an upper semicontinuous function $g:[0,1]\rightarrow C([0,1])$ such that $(x,y)\in G(g)\subseteq G(f)$. To that end, suppose $(x,y)\in G(f)$. For each $n\in\mathbb{N}$, define $\mathcal{G}_{n}$ to be the collection of set-valued functions $g_{n}$ that satisfy the following: let $0\leq i< 2^{n}$ such that $\frac{i}{2^{n}}\leq x\leq \frac{i+1}{2^{n}}$ and $C_{n,i}$ be the component of $G(f)\cap V_{\left[\frac{i}{2^{n}},\frac{i+1}{2^{n}}\right]}$ containing $(x,y)$. By Theorem \ref{Intersection}, $C_{n,i}\cap \left(\left\{\frac{i+1}{2^{n}}\right\}\times[0,1]\right)\neq\emptyset$. Choose a component $C_{n,i+1}$ of $G(f)\cap V_{\left[\frac{i+1}{2^{n}},\frac{i+2}{2^{n}}\right]}$ such that $C_{n,i}\cap C_{n,i+1}\neq\emptyset$. Then by Theorem \ref{Intersection}, $C_{n,i+1}\cap \left(\left\{\frac{i+2}{2^{n}}\right\}\times[0,1]\right)\neq\emptyset$. Continuing inductively, for each $i<j<2^{n}$, we may choose a component $C_{n,j}$ of $G(f)\cap V_{\left[\frac{j}{2^{n}},\frac{j+1}{2^{n}}\right]}$ such that $C_{n,j-1}\cap C_{n,j}\neq\emptyset$ and $C_{n,j}\cap \left(\left\{\frac{j+1}{2^{n}}\right\}\times[0,1]\right)\neq\emptyset$. By a similar argument, for $0\leq j<i$, we may choose a component $C_{n,j}$ of $G(f)\cap V_{\left[\frac{j}{2^{n}},\frac{j+1}{2^{n}}\right]}$ such that $C_{n,j}\cap C_{n,j+1}\neq\emptyset$ and $C_{n,j}\cap \left(\left\{\frac{j}{2^{n}}\right\}\times[0,1]\right)\neq\emptyset$. So for $0\leq j<2^{n}$ there are components $C_{n,j}$ of $G(f)\cap V_{\left[\frac{j}{2^{n}},\frac{j+1}{2^{n}}\right]}$ such that $C_{n,j}\cap C_{n,k}\neq\emptyset$ if and only if $|j-k|\leq 1$. Define $G(g_{n})=\bigcup_{j< 2^{n}}C_{n,j}$. Then $g_{n}$ is upper semicontinuous and has a connected graph. Since for any $i<2^{n}$ and component $C_{n,i}$ of $G(f)\cap V_{\left[\frac{i}{2^{n}},\frac{i+1}{2^{n}}\right]}$ there is some $g_{n}\in\mathcal{G}_{n}$ with $C_{n,i}\subseteq G(g_{n})$, $\bigcup_{g_{n}\in\mathcal{G}_{n}}G(g_{n})=G(f)$.

Let $\{G(g_{n})\}_{n\in\mathbb{N}}$ be a sequence of subcontinua of $G(f)$ where $g_{n}\in\mathcal{G}_{n}$ for each $n$. By Lemma \ref{Hyperspace}, there is a convergent subsequence $\{G(g_{n_{k}})\}_{k\in\mathbb{N}}$ such that $\lim G(g_{n_{k}})$ is a subcontinuum of $G(f)$. Define $g$ by $G(g)=\lim G(g_{n_{k}})$. Then $g$ is usc and has a connected graph. Furthermore $(x,y)\in G(g)$, as $(x,y)\in G(g_{n_{k}})$ for every $k$. Since $\left(\{0\}\times g_{n_{k}}(0)\right)\in 2^{[0,1]^{2}}$ for each $k$, there is a convergent subsequence of $\left\{\{0\}\times\{g_{n_{k}}(0)\}\right\}$. Then $\overline{\lim} \left(\{0\}\times g_{n_{k}}(0)\right)$ is a nonempty subset of $\overline{\lim} G(g_{n_{k}})=\lim G(g_{n_{k}})=G(g)$. Thus $g(0)\neq\emptyset$. Similarly $g(1)\neq\emptyset$. As $G(g)$ is connected, $g$ is a usc function on $[0,1]$. Let $\mathcal{F}$ consist of all such functions $g$ for any $(x,y)\in G(f)$ and any sequence $\{G(g_{n})\}_{n\in\mathbb{N}}$ where $g_{n}\in\mathcal{G}_{n}$ for each $n$.

In order to show $g(x)$ is connected for each $x\in[0,1]$, let $(x,y),(x,y')\in G(g_{\alpha})$. Then for each $k$ there are points $(x_{n_{k}},y_{n_{k}}),(x_{n_{k}}',y_{n_{k}}')\in G(g_{n_{k}})$ such that $(x_{{n}_{k}},y_{n_{k}})\rightarrow (x,y)$ and $(x_{n_{k}}',y_{n_{k}}')\rightarrow (x,y')$. Let $i_{n_{k}}$ and $j_{n_{k}}$ be the largest and smallest integers respectively such that $x_{n_{k}},x_{n_{k}}'\in\left[\frac{i_{n_{k}}}{2^{n_{k}}},\frac{j_{n_{k}}}{2^{n_{k}}}\right]$. Note that $\frac{i_{n_{k}}}{2^{n_{k}}},\frac{j_{n_{k}}}{2^{n_{k}}}\rightarrow x$ because $x_{n_{k}},x_{n_{k}}'\rightarrow x$. By the construction of the $g_{n_{k}}'s$, $A_{n_{k}}=G(g_{n_{k}})\cap V_{\left[\frac{i_{n_{k}}}{2^{n_{k}}},\frac{j_{n_{k}}}{2^{n_{k}}}\right]}$ is a subcontinuum containing $(x_{n_{k}},y_{n_{k}})$ and $(x_{n_{k}}',y_{n_{k}}')$. So $(x,y),(x,y')\in\underline{\lim}\ A_{n_{k}}$. Then by Theorem \ref{liminf}, $\overline{\lim}\ A_{n_{k}}$ is a connected subset of $\lim G(g_{n_{k}})=G(g)$ containing $(x,y)$ and $(x,y')$. Since $\frac{i_{n_{k}}}{2^{n_{k}}},\frac{j_{n_{k}}}{2^{n_{k}}}\rightarrow x$, $\overline{\lim}A_{n_{k}}\subseteq\left(\{x\}\times f(x)\right)$. As $y$ and $y'$ are arbitrary elements of $g(x)$, $g(x)$ is connected. Then by Corollary \ref{PointImage}, $g$ has the Weak Intermediate Value Property.

Next we show that $G(f)=\bigcup_{\alpha\in\Lambda}G(g)$. That $\bigcup_{\alpha\in\Lambda}G(g_{\alpha})\subseteq G(f)$ follows from the fact that $G(g_{\alpha})\subseteq G(f)$ for each $\alpha$. To show $G(f)\subseteq\bigcup_{\alpha\in\Lambda}G(g_{\alpha})$, let $(x,y)\in G(f)$. Then for each $n\in\mathbb{N}$, there is a usc function $g_{n}\in\mathcal{G}_{n}$ such that $(x,y)\in G(g_{n})$. Then there is a convergent subsequence $\{G(g_{n_{k}})\}$. Let $g_{\alpha}\in\mathcal{F}$ be such that $G(g_{\alpha})=\lim\ G(g_{n_{k}})$. Then $(x,y)\in G(g_{\alpha})$.

\end{proof}

\begin{theorem}[Nall, \cite{Nall}] \label{Nall}
Suppose $X$ is a compact metric space, and $\{F_{\alpha}\}_{\alpha\in\Lambda}$ is a collection of closed subsets of $X\times X$ such that for each $x\in X$ and each $\alpha\in\Lambda$, the set $\{y\in X:(x,y)\in F_{\alpha}\}$ is nonempty and connected, and such that $F=\bigcup_{\alpha\in\Lambda}F_{\alpha}$ is a closed connected subset of $X\times X$ such that for each $y\in X$, the set $\{x\in X:(x,y)\in F\}$ is nonempty. Then $\invlim \ F$ is connected.
\end{theorem}

Using the description of usc functions with the Weak Intermediate Value Property from Theorem \ref{Union} together with Theorem \ref{Nall}, we may provide an alternate proof for Theorem \ref{InvLim} in the context of inverse limits with a single bonding map. Suppose $f:[0,1]\rightarrow2^{[0,1]}$ is usc, surjective, and has the Weak Intermediate Value Property and that $G(f)$ is connected. Then there is a collection $\{g_{\alpha}\}_{\alpha\in\Lambda}$ of usc functions $g_{\alpha}:[0,1]\rightarrow C([0,1])$ as described by Theorem \ref{Union}. Letting $F_{\alpha}=G(g_{\alpha})$ for each $\alpha\in\Lambda$, we obtain a collection of subcontinua of $[0,1]^{2}$ such that for each $x\in[0,1]$, $\{y\in[0,1]:(x,y)\in F_{\alpha}\}$ is a nonempty continuum and $F=\bigcup_{\alpha\in\Lambda}G(g_{\alpha})=G(f)$. As $f$ is surjective, for each $y\in[0,1]$, $\{x\in[0,1]:(x,y)\in F\}\neq\emptyset$. Thus $\invlim \ F=\invlim\{[0,1],f\}$ is connected by Theorem \ref{Nall}.

\begin{theorem}\label{IVPIntersection}
Let $f:[0,1]\rightarrow 2^{[0,1]}$ be usc. Then the following are equivalent.
\begin{enumerate}
 \item $f$ has the Intermediate Value Property.
 \item For all $a\leq b$, $G(f)\cap V_{[a,b]}$ is connected and $G(f)\cap V_{[a,b]}=\overline{G(f)\cap V_{(a,b)}}$.
 \end{enumerate}
\end{theorem}

\begin{proof}
$(1\Rightarrow2)$ Suppose $f$ has the Intermediate Value Property. In order to show that for all $a\leq b$ $G(f)\cap V_{[a,b]}$ is connected, we first show that $f(x)$ is connected for each $x\in[0,1]$. By way of contradiction, suppose there is an $x_{1}$ such that $f(x_{1})$ is disconnected. Let $U_{1}$ and $U_{2}$ be two intervals open in $[0,1]$ such that $f(x_{1})\subseteq U_{1}\cup U_{2}$, $ f(x_{1})\cap U_{1}\neq\emptyset$, $f(x_{1})\cap U_{2}\neq\emptyset$ and $\overline{U}_{1}\cap \overline{U}_{2}=\emptyset$.

Since $f$ is usc, there is an open set $V\ni x_{1}$ such that if $x\in V$, then $f(x)\subseteq U_{1}\cup U_{2}$. Let $x_{2}\in V\setminus\{x_{1}\}$ and choose $y_{2}\in f(x_{2})$. Then $y_{2}\in U_{1}$ or $y_{2}\in U_{2}$. Without loss of generality, suppose $y_{2}\in U_{2}$. Let $y_{1}\in f(x_{1})\cap U_{1}$. Then there is some $y\in[0,1]\setminus(U_{1}\cup U_{2})$ strictly between $y_{1}$ and $y_{2}$. Let $x$ be between $x_{1}$ and $x_{2}$. Then since $x\in V$, $f(x)\subseteq U_{1}\cup U_{2}$. So $y\notin f(x)$, contradicting the assumption that $f$ has the Intermediate Value Property. Thus $f(x)$ is connected for all $x\in[0,1]$.

Let $K_{1}$ and $K_{2}$ be components of $G(f)\cap V_{[a,b]}$. Since $f$ has the Intermediate Value Property and therefore the Weak Intermediate Value Property, $K_{1}\cap \left(\{a\}\times [0,1]\right)\neq\emptyset$ and $K_{2}\cap \left(\{a\}\times[0,1]\right)\neq\emptyset$ by Theorem \ref{Intersection}. But because $f(a)$ is connected, $\{a\}\times f(a)$ is connected and intersects both $K_{1}$ and $K_{2}$. So $K_{1}\cap K_{2}\neq\emptyset$ as $K_{1}$ and $K_{2}$ are components. Hence $K_{1}=K_{2}$, making $G(f)\cap V_{[a,b]}$ connected. 

Next we show $G(f)\cap V_{[a,b]}=\overline{G(f)\cap V_{(a,b)}}$. If $f(a)$ is a singleton, then $(a,f(a))\in\overline{G(f)\cap V_{(a,b)}}$ as $f$ is usc. Similarly if $f(b)$ is a singleton, then $(b,f(b))\in\overline{G(f)\cap V_{(a,b)}}$. Suppose $f(a)$ is nondegenerate. Let $\{(a_{n},z_{n})\}_{n\in\omega}$ be a sequence in $G(f)\cap V_{(a,b)}$ such that $a_{n}\rightarrow a$. Taking subsequences if necessary, we may choose the sequence such that $\{z_{n}\}_{n\in\omega}$ converges to some $z\in f(a)$. So $(a,z)\in\overline{G(f)\cap V_{(a,b)}}$. 

Let $y\in f(a)\setminus\{z\}$ and let $\epsilon>0$ such that $|z-y|>\epsilon$. Since $z_{n}\rightarrow z$, there is $N\in\mathbb{N}$ such that if $n\geq N$, then $y$ is not between $z_{n}$ and $z$. Let $y_{\epsilon}$ be strictly between $y$ and $z_{n}$ for all $n\geq N$ such that $|y-y_{\epsilon}|<\epsilon$. Since $f$ has the Intermediate Value Property, for each $n\geq N$ there is some $x_{n}\in(a,a_{n})$ such that $y_{\epsilon}\in f(a_{n})$. Then $a_{n}\rightarrow a$ and $(a,y_{\epsilon})\in \overline{G(f)\cap V_{(a,b)}}$. By the same argument, for all $0<\delta<\epsilon$ there is some $y_{\delta}$ such that $(a,y_{\delta})\in \overline{G(f)\cap V_{(a,b)}}$. Since $d((a,y),(a,y_{\delta}))=\delta$, $(a,y)\in\overline{G(f)\cap V_{(a,b)}}$. Therefore $\{a\}\times f(a)\subseteq \overline{G(f)\cap V_{(a,b)}}$. By a similar argument, $\{b\}\times f(b)\subseteq \overline{G(f)\cap V_{(a,b)}}$. Thus $\overline{G(f)\cap V_{(a,b)}}=G(f)\cap V_{[a,b]}$.

$(2\Rightarrow1)$ The converse follows by contradiction. Suppose for all $a\leq b$, $G(f)\cap V_{[a,b]}$ is connected and $G(f)\cap V_{[a,b]}=\overline{G(f)\cap V_{(a,b)}}$, but $f$ does not have the Intermediate Value Property. Then there is some distinct $x_{1}, x_{2}$, $y_{1}\in f(x_{1})$, $y_{2}\in f(x_{2})$, and $y$ strictly between $y_{1}$ and $y_{2}$ such that $y\notin f(x)$ for all $x$ strictly between $x_{1}$ and $x_{2}$. There are four cases depending on which of $x_{1}$ and $x_{2}$ is larger and which of  $y_{1}$ and $y_{2}$ is larger. Suppose $x_{1}<x_{2}$ and $y_{1}<y_{2}$. The proofs for the remaining cases  are similar.

Since $G(f)\cap V_{[x_{1},x_{2}]}=\overline{G(f)\cap V_{(x_{1},x_{2})}}$, there are $x_{1}<a'<b'<x_{2}$, $y_{a'}\in f(a')$ with $y_{a'}<y$ , and $y_{b'}\in f(b')$ with $y_{b'}>y$. So $y$ is strictly between $y_{a'}$ and $y_{b'}$, but $y\notin f(x)$ for any $x\in[a',b']$. Thus $G(f)\cap V_{[a',b']}$ is disconnected, a contradiction.
\end{proof}

\section{Examples}

The first three examples demonstrate that Theorem \ref{InvLim} is sharp in that the conditions on the bonding functions cannot be dropped. Example \ref{Sufficient} demonstrates that the inverse limit may be connected even if the bonding functions do not have the Weak Intermediate Value Property.

\begin{example}[Nall]\label{NoIVP}
Let $f:[0,1]\rightarrow2^{[0,1]}$ be given by \[f(x)=\left\{\begin{array}{lr}
\frac{1}{3}x & 0\leq x<\frac{1}{2} \\
\left\{\frac{1}{3}x, 2x-1\right\} & \frac{1}{2}\leq x\leq 1
\end{array}\right.\]
Then $f$ is an upper semicontinuous surjective function such that $G(f)$ is connected that does not have the Weak Intermediate Value Property. Yet for each $x_{1}\in[0,1]$ for all $y_{1}\in f(x_{1})$ and $x_{2}\geq x_{1}$, there is some $y_{2}\in f(x_{2})$ such that if $y$ is between $y_{1}$ and $y_{2}$, there is some $x$ between $x_{1}$ and $x_{2}$ such that $y\in f(x)$. Let $(x_{1},y_{1})=\left(\frac{1}{2},0\right)$ and $x_{2}=\frac{1}{4}$. Since $f\left(\frac{1}{4}\right)=\left\{\frac{1}{12}\right\}$, $y_{2}$ must be $\frac{1}{12}$. But $\frac{1}{24}\notin f(x)$ for any $x\in\left[\frac{1}{4},\frac{1}{2}\right]$. So $f$ does not have the Weak Intermediate Value Property. Further, Nall shows in \cite{Nall} that $G(f^{2})$ has an isolated point at $(1,0)$, so $\invlim\{[0,1],f\}$ is not connected.
\end{example}

\begin{figure}[H]
\begin{center}
 \begin{tikzpicture}[x=3.5cm, y=3.5cm] 
\draw (0,1) -- (1,1);
\draw (1,0) -- (1,1);
\draw (0,0) -- (1,0);
\draw (0,0) -- (0,1);
\draw[blue,domain=0:1,samples=1000] plot (\x, {0.5*\x});
\draw[blue,domain=0.5:1,samples=1000] plot (\x, {2*\x-1});
\node[label={left:{\tiny{$(0,0)$}}},circle,fill,blue,inner sep=1pt] at (0,0) {};
\node[label={below:{\tiny{$(1/2,0)$}}},circle,fill,blue,inner sep=1pt] at (0.5,0) {};
\node[label={right:{\tiny{$(1,1/2)$}}},circle,fill,blue,inner sep=1pt] at (1,0.5) {};
\node[label={right:{\tiny{$(1,1)$}}},circle,fill,blue,inner sep=1pt] at (1,1) {};

    \draw (1.5,0) -- (2.5,0);
    \draw (1.5,0) -- (1.5,1);
    \draw (1.5,1) -- (2.5,1);
    \draw (2.5,0) -- (2.5,1);
\node[label={below:{\tiny{$(0,0)$}}},circle,fill,blue,inner sep=1pt] at (1.5,0) {};
\node[label={below:{\tiny{$(1/2,0)$}}},circle,fill,blue,inner sep=1pt] at (2,0) {};
\node[label={below right:{\tiny{$(3/4,0)$}}},circle,fill,blue,inner sep=1pt] at (2.25,0) {};
\node[label={right:{\tiny{$(1,1/2)$}}},circle,fill,blue,inner sep=1pt] at (2.5,0.5) {};
\node[label={right:{\tiny{$(1,1)$}}},circle,fill,blue,inner sep=1pt] at (2.5,1) {};
\node[label={right:{\tiny{$(1,1/4)$}}},circle,fill,blue,inner sep=1pt] at (2.5,0.25) {};
\node[label={right:{\tiny{$(1,0)$}}},circle,fill,blue,inner sep=1pt] at (2.5,0) {};
    \draw[blue] (1.5,0) -- (2.5,0.25);
  \draw[blue] (2,0) -- (2.5,0.5);
  \draw[blue] (2.25,0) -- (2.5,1);

\end{tikzpicture}
 \end{center}
 \caption{$G(f)$ (left) and $G(f^{2})$ (right) from Example \ref{NoIVP}} \label{Fig1}
\end{figure}
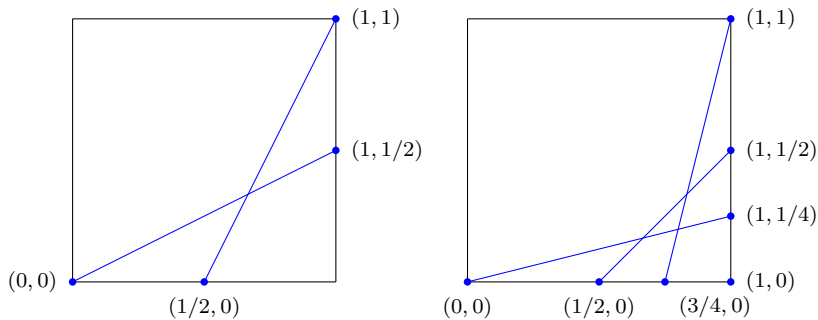

\begin{example}\label{DisconnectedGraph}
The function $f:[0,1]\rightarrow 2^{[0,1]}$ given by $f(x)=\left\{\frac{2}{3}x,\frac{2}{3}x+\frac{1}{3}\right\}$ is an upper semicontinuous function that is surjective and satisfies the Weak Intermediate Value Property. But $G(f)$ is not connected, therefore $\invlim\{[0,1],f\}$ is not connected.
\end{example}

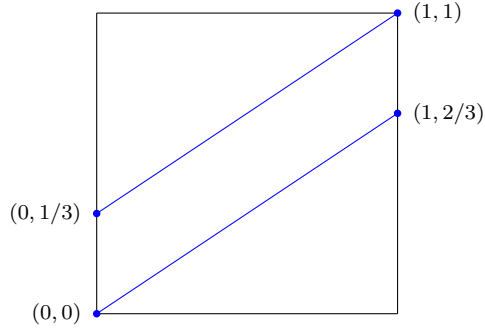
\begin{figure}[H]
\begin{center}
 \begin{tikzpicture}[x=4cm, y=4cm] 
\draw (0,1) -- (1,1);
\draw (1,0) -- (1,1);
\draw (0,0) -- (1,0);
\draw (0,0) -- (0,1);
\draw[blue,domain=0:1,samples=1000] plot (\x, {2/3*\x});
\draw[blue,domain=0:1,samples=1000] plot (\x, {2/3*\x+1/3});
\node[label={left:{\tiny{$(0,0)$}}},circle,fill,blue,inner sep=1pt] at (0,0) {};
\node[label={left:{\tiny{$(0,1/3)$}}},circle,fill,blue,inner sep=1pt] at (0,1/3) {};
\node[label={right:{\tiny{$(1,2/3)$}}},circle,fill,blue,inner sep=1pt] at (1,2/3) {};
\node[label={right:{\tiny{$(1,1)$}}},circle,fill,blue,inner sep=1pt] at (1,1) {};
\end{tikzpicture}
 \end{center}
 \caption{$G(f)$ from Example \ref{DisconnectedGraph}} \label{Fig2}
 \end{figure}

\begin{example}\label{NotSurjective}
The function $f:[0,1]\rightarrow 2^{[0,1]}$ given by $f(x) = \left\{\frac{1}{4},\frac{3}{4}x+\frac{1}{4}\right\}$ is upper semicontinuous, satisfies the Weak Intermediate Value Property, and has a connected graph. But $f$ is not surjective. Specifically if $y\in\left[0,\frac{1}{4}\right)$, there is no $x$ with $y\in f(x)$. Thus $\invlim\{[0,1],f\} = \invlim\left\{\left[\frac{1}{4},1\right],f|_{\left[\frac{1}{4},1\right]}\right\}$ which is not connected.
\end{example}

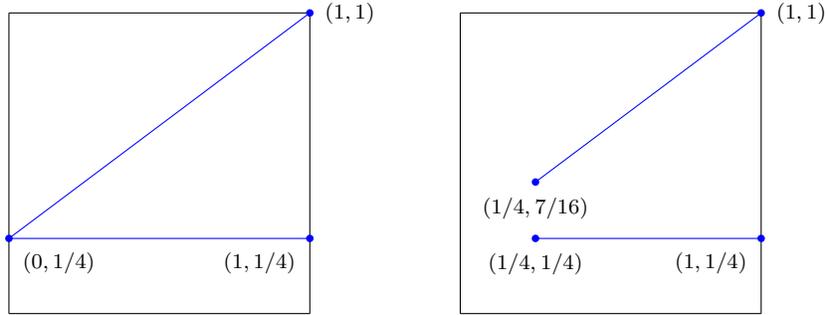
\begin{figure}[H]
  \begin{center}
      \begin{tikzpicture}[x=4cm, y=4cm] 
    \draw (0,0) -- (1,0);
    \draw (0,0) -- (0,1);
    \draw (0,1) -- (1,1);
    \draw (1,0) -- (1,1);
  \draw[blue,domain=0:1,samples=1000] plot (\x, {1/4});
  \draw[blue,domain=0:1,samples=1000] plot (\x, {3*\x/4+(1/4)});
\node[label={below left:{\tiny{$(1,1/4)$}}},circle,fill,blue,inner sep=1pt] at (1,0.25) {};
\node[label={right:{\tiny{$(1,1)$}}},circle,fill,blue,inner sep=1pt] at (1,1) {};
\node[label={below right:{\tiny{$(0,1/4)$}}},circle,fill,blue,inner sep=1pt] at (0,0.25) {};
    \draw (1.5,0) -- (2.5,0);
    \draw (1.5,0) -- (1.5,1);
    \draw (1.5,1) -- (2.5,1);
    \draw (2.5,0) -- (2.5,1);
\node[label={below:{\tiny{$(1/4,1/4)$}}},circle,fill,blue,inner sep=1pt] at (1.75,0.25) {};
\node[label={below:{\tiny{$(1/4,7/16)$}}},circle,fill,blue,inner sep=1pt] at (1.75,7/16) {};
\node[label={below left:{\tiny{$(1,1/4)$}}},circle,fill,blue,inner sep=1pt] at (2.5,0.25) {};
\node[label={right:{\tiny{$(1,1)$}}},circle,fill,blue,inner sep=1pt] at (2.5,1) {};
    \draw[blue,domain=1.75:2.5,samples=1000] plot (\x, {1/4});
  \draw[blue,domain=1.75:2.5,samples=1000] plot (\x, {3*(\x-1.5)/4+(1/4)});
  \end{tikzpicture}
     \end{center} 
     \caption{$G(f)$ (left) and $G(f|_{\left[\frac{1}{4},1\right]})$ (right) for Example \ref{NotSurjective}} \label{Fig3}
\end{figure}

\begin{example}[Ingram]\label{Sufficient}
Let $f:[0,1]\rightarrow2^{[0,1]}$ be given by \[f(x)=\left\{\begin{array}{lr}
\{0,x\} & 0\leq x\leq\frac{1}{4} \\
0 & \frac{1}{4}\leq x\leq 1 \\
\left[0,1\right] & x=1
\end{array}\right.\]
Then $f$ does not have the Weak Intermediate Value Property but $\invlim\{[0,1],f\}$ is connected. If $(x_{1},y_{1})=\left(\frac{1}{4},\frac{1}{4}\right)$ and $x_{2}=\frac{1}{2}$, then $y_{2}=0$. But $\frac{1}{8}\notin f(x)$ for any $x\in\left[\frac{1}{4},\frac{1}{2}\right]$. Ingram gives a proof that $\invlim\{[0,1],f\}$ is connected can be found in Example 2.9 pg. 24 of \cite{Ingram}.
\end{example}

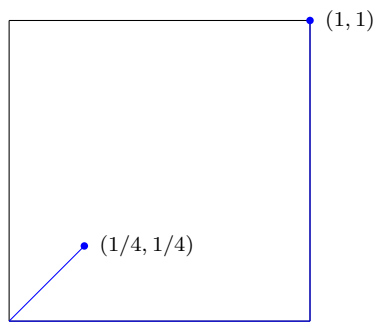
\begin{figure}[H]
  \begin{center}
      \begin{tikzpicture}[x=4cm, y=4cm] 
    \draw (0,0) -- (1,0);
    \draw (0,0) -- (0,1);
    \draw (0,1) -- (1,1);
    \draw (1,0) -- (1,1);
    \draw[blue] (0,0) -- (0.25,0.25);
  \draw[blue] (0,0) -- (1,0);
  \draw[blue] (1,0) -- (1,1);
  \node[label={right:{\tiny{$(1/4,1/4)$}}},circle,fill,blue,inner sep=1pt] at (0.25,0.25) {};
  \node[label={right:{\tiny{$(1,1)$}}},circle,fill,blue,inner sep=1pt] at (1,1) {};
  \end{tikzpicture}
     \end{center}
     \caption{$G(f)$ for Example \ref{Sufficient}}\label{Fig4}
\end{figure}

\section{Acknowledgements}
This research did not receive any specific grant from funding agencies in the public, commercial, or not-for-profit sectors. The author is grateful to David Ryden and Jonathan Meddaugh for their assistance in proofreading this paper.

\end{document}